\documentclass[11pt,leqno]{article}
\usepackage{amsfonts}
\usepackage{mathrsfs}
\usepackage{epsfig, setspace, subfigure}
\usepackage{a4}
\usepackage{amssymb}
\usepackage{latexsym}
\usepackage{subfigure}

\usepackage{amsbsy}
\usepackage{amsthm}
\usepackage{pst-plot}
\usepackage{amsmath}
\usepackage{hhline}
\usepackage{pifont}
\usepackage{enumerate,amsfonts}
\usepackage{indentfirst}
\usepackage{subfigure}

\usepackage{graphicx}

\newtheorem{theorem}{\bf Theorem}[section]

\newtheorem{remark}{\bf Remark}[section]
\newtheorem{lemma}{\bf Lemma}[section]

\newcommand{\R}{{I\!\!R}}

\newcommand{\br}{{\bf r}}

\newcommand{\bv}{{\bf v}}

\newcommand{\bV}{{\bf V}}
\newcommand{\bH}{{\bf H}}

\newcommand{\bp}{\boldsymbol{\phi}}
\newcommand{\bs}{\boldsymbol{\sigma}}
\newcommand{\bS}{\boldsymbol{\Sigma}}

\newcommand{\es}{e_{\boldsymbol{\sigma}}}
\newcommand{\se}{\setcounter{equation}{0}}
\newsavebox{\savepar}
\newcommand{\xibs}{\xi_{\bs}}
\newcommand{\xiu}{\xi_u}
\newcommand{\xiut}{\xi_{u,t}}

\begin{document}
\title{\bf\Large {\it A posteriori} error estimates for mixed finite element Galerkin
              approximations of  second order linear hyperbolic equations}

%\author{}
\author{Samir Karaa{\footnote{
 Department of Mathematics and Statistics, Sultan Qaboos University,
 P. O. Box 36, Al-Khod 123,
 Muscat, Oman. Email: skaraa@squ.edu.om} }
 and
 Amiya K. Pani{\footnote{
 Department of Mathematics, Industrial Mathematics Group, Indian
 Institute of Technology Bombay,
 Powai, Mumbai-400076.
 Email: akp@math.iitb.ac.in}}}
\date{}
\maketitle

\begin{abstract}
In this article, {\it a posteriori} error analysis 
for mixed finite element  Galerkin approximations of
second order linear hyperbolic equations is  discussed. Based on
mixed elliptic reconstructions and an integration tool, which is a variation of Baker's
technique introduced earlier by G. Baker (SIAM J. Numer. Anal., 13 (1976), 564–-576) 
in the context of {\it a priori} estimates for a second order wave equation,
{\it a posteriori} error estimates of the displacement in $L^{\infty}(L^2)$-norm for the semidiscrete scheme
are derived. %under minimal regularity. 
Finally, a first order implicit-in-time  discrete scheme is analyzed and  
{\it a posteriori}
error estimators are established.
\end{abstract}

\vspace{0.5cm}

{\bf Key words}. second order linear wave equation, mixed finite element
methods, mixed elliptic reconstructions, semidiscrete method, first order implicit completely discrete scheme,
{\it a posteriori } error estimates.

\section{\bf Introduction}%----------------------------------------------ONE

In this paper, we discuss {\it a posteriori } error estimates for  mixed
finite  element Galerkin approximations to the following class of second order
linear hyperbolic problems:
\begin {eqnarray} \label{par}
 u_{tt}- \nabla\cdot(A\nabla u) = f && \mbox{in}~~\Omega\times~(0,T],  \\
 u|_{\partial\Omega}= 0&& u|_{t=0}= u_0  \quad \mbox{and} \quad  u_{t}|_{t=0}=u_1 .\label{ic-bc}
\end{eqnarray}
Here, $\Omega\subset \R^2$ is a bounded
 polygonal domain with boundary $\partial\Omega,$ $0<T<
\infty,$ $u_t=\frac{\partial u}{\partial t}$ and
$A(x)=(a_{ij}(x))_{1\le i,j \le 2}$ is a symmetric and uniformly
positive definite matrix. All the coefficients $a_{ij}$'s are smooth functions
of $x$ with uniformly bounded derivatives in $\bar{\Omega}.$ Moreover,
the initial functions $u_0=u_0(x)$, $u_1=u_1(x)$ and the forcing function
$f=f(x,t)$ are assumed to be smooth functions in their
respective domains.

In recent years, there has been a growing demand for designing reliable and
efficient space-time algorithms for the numerical computation of time dependent 
partial differential equations. Most of these algorithms are based on  {\it a posteriori} 
error estimators, which provide appropriate tools for adaptive mesh refinements. 
For elliptic boundary value problems, {\it a posteriori} error estimates are 
well developed (see, \cite{AO-2000,Rud-2013}).  Adaptivity with {\it a posteriori} 
error control for parabolic problems has also been an active research  area  for the last two decades 
(cf. \cite{EJ-91, V-98(b),JNT-90,NSV-00,BB-82(a),BB-82(b),BO-01} 
and references, therein). For the time discretization, some results on {\it a posteriori} error estimations 
for abstract first order evolution problems  are available in the literature  
(cf. \cite{AMN-2006,GL-08,LM-06,MN03,NSV-00}). 
%In particular, with the help of higher order
%appropriate reconstructions of the approximate solutions, optimal {\it a posteriori} error estimators for
%some time discretization methods are established in \cite{AMN-2006,GL-08,LM-07,LM-06,MN03,NSV-00,NSV}.
%Furthermore, a posteriori superconvergence
%estimates for the error at the nodes for Galerkin and Runge-Kutta methods were derived in [4]. 

In the context of second order wave equations, only few results are available on {\it a posteriori} error 
analysis, 
see, \cite{Johnson-1993,Adjerid-2002,BR-2001,BR-1999,Suli-2005,Picasso-2010}. Further, it is observed that
the design and implementation of adaptive algorithms for these 
equations based on rigorous {\it a posteriori} error estimators are less complete compared to
elliptic and  parabolic equations.
%Although some of the results obtained for first order evolution equations apply (directly or after
%appropriate modifications) to the second order wave equation, when written as a first order
%system, this framework does not cover popular two-step implicit or explicit time-discretisation methods.
Based on a space-time finite element discretization with  basis functions
being continuous in space and discontinuous in time, {\it a priori} and {\it a posteriori} error estimates 
for second order linear wave equations are proved in \cite{Johnson-1993}. 
Asymptotically exact {\it a posteriori} estimates for the standard finite element method  are proposed 
and analyzed  in \cite{Adjerid-2002,Adjerid-2006} by solving a set of local elliptic problems. % ,  but only in some specific situations.
The recent results in \cite{Suli-2005,GLM-2013} cover only  first order time discrete schemes. 
In \cite{Suli-2005}, the second order wave equation is written as a first order system and a first order 
implicit backward Euler scheme in time is used with continuous piecewise affine finite elements in space. 
Further, rigorous {\it a posteriori} bounds have been established using energy arguments 
and adaptive algorithms based on the {\it a posteriori} bounds are discussed. 
In  \cite{GLM-2013}, based on Baker's technique {\it a posteriori} bounds 
are derived for the semidiscrete in $L^\infty(L^2)$-norm  and for first order
implicit-in-time fully discrete schemes in  $\ell^\infty(L^2)$-norm. The fully  discrete analysis  relies crucially on a novel time reconstruction
satisfying a local vanishing-moment property, and on a space reconstruction technique used earlier in 
\cite{MN03} for parabolic problems. In \cite{BR-2001},  an adaptive algorithm in space and time
 which  is  based on Galerkin space-time discretizations leading to Newmark
scheme is analyzed. Further, goal oriented {\it a posteriori} error estimates are derived and 
some numerical results are provided to demonstrate the efficiency of error estimators. 
In \cite{Picasso-2010}, the author has studied an anisotropic {\it a posteriori} error estimate for 
a finite element discretization of a two dimensional wave equation. 
The estimate is derived in the $L^2 (0,T,H^1(\Omega))$-norm and it turns out to be sharp on anisotropic meshes. %, whenever the spatial discretization error is predominant. 

For higher order time reconstruction for abstract second order evolution equations, one may refer to the recent papers 
\cite{HLT-2013,GLMV-2015}.  In \cite{HLT-2013},  %authors have discussed 
an  adaptive time stepping Galerkin method is analyzed for second order evolution problems.
% {\it a posteriori} error analysis is established. 
 Based on the energy approach and the duality argument,
optimal order {\it a posteriori} error estimates and {\it a posteriori} nodal superconvergence results have
been derived. An adaptive time stepping strategy is discussed and some numerical experiments are 
conducted to assess the effectiveness of the proposed scheme. %adaptive time stepping methods. 
In a recent work \cite{GLMV-2015}, %authors have discussed 
second order explicit and implicit two-step
time discretization schemes such as leap-frog and cosine methods are discussed and {\it a posteriori} estimates 
using a novel time reconstruction are derived. Further, %they presented 
some numerical experiments are conducted  to confirm their theoretical findings.

For space-time adaptivity, the finite element discretization depends on
the space-time variational formulation and its error indicators include both
space and time errors. Recently, attempts have been made to exploit elliptic reconstruction to 
prove optimal {\it a posteriori } error estimates in finite element methods for parabolic problems \cite{MN03}.  
In fact, the role of the elliptic reconstruction operator in {\it a posteriori } estimates is quite similar 
to the role played by elliptic projection introduced earlier by  Wheeler \cite{W-73} for recovering
optimal {\it a priori} error estimates of  finite element Galerkin approximations to parabolic problems.
This analysis is, further, developed for
completely discrete scheme based on backward Euler method \cite{LM-06},
for maximum norm estimates \cite{DLM-09} and for discontinuous Galerkin methods
for parabolic problems \cite{GL-08}. In recent works \cite{MNP-2012} and \cite{LM-2011}, the analysis is further 
extended  to mixed FE Galerkin methods applied to parabolic problems. 

%See  also
%\cite{LM-2011} for a different setting
%another {\it a posteriori} error analysis with for a mixed finite element
%formulation.

In this article, an {\it a posteriori} analysis is discussed for mixed finite
element Galerkin approximations  of  a class of second order linear hyperbolic problems. 
%Mixed methods for hyperbolic problems  offer a simultaneous approximation 
%for the displacement and its velocity, which are important in many applications.
%In a mixed finite element formulation both displacements and stresses are
%approximated simultaneously. This approach provides higher-order approximations to the
%stresses. 
One notable advantage of mixed finite element scheme is that it offers a simultaneous approximations 
of displacements and stresses, resulting in better convergences rates for the stress variable.
This property is important in applications such as in  
the modeling boundary controllability of the wave equation.
In the first part of this article, a semidiscrete scheme is derived using mixed finite element method 
in spatial direction,  while keeping time variable constant. 
Based on mixed elliptic reconstructions presented in \cite{MNP-2012}, which depend 
explicitly on residuals and a time integration tool, a variant of {Baker's} technique,  
{\it a posteriori } error estimates in $L^{\infty}(L^2)$-norm are derived   
for the displacement $u.$ %under minimal regularity.
For the time discretization, the time discrete scheme with the time-reconstruction proposed in \cite{GLM-2013}
is applied and then, using summation tool, {\it a posteriori } error estimators in $\ell^{\infty}(L^2)$-norms are 
developed. Compared to \cite{GLM-2013}, our analysis is not only for mixed finite element method, but also 
it differs from the analysis of \cite{GLM-2013} in 
the sense that a time integration tool is used for deriving  $L^{\infty}(L^2)$ {\it a posteriori } estimators,
as against the time testing procedure of  Baker \cite{Baker-1976} used in \cite{GLM-2013}.

The outline of this article is as follows. {In Section \ref{sec:2},  we introduce both 
weak primal and mixed formulations for the hyperbolic problem (\ref{par})-(\ref{ic-bc})  and 
establish their equivalence.} Section \ref{sec:3} deals with mixed elliptic reconstruction 
techniques proposed in \cite{MNP-2012} and 
{\it a posteriori } estimates  for the semidiscrete problem for both
displacement $u$ and its stress $\bs $ in 
$L^{\infty}(L^2)$-norms are derived. Based on a first order
backward differencing implicit method, a completely discrete scheme  is proposed and 
related  {\it a posteriori } error estimators are established in Section \ref{sec:4}. 
Finally,  results are summarized in Section {\ref{sec:5}} with a brief outline on future work.

%**************************************************************************
\section {\bf On primal and mixed formulations} \label{sec:2}
% estimates for semidiscrete mixed method}
%--------------------FOURTH
\se
%is considered and  {\it a posteriori } error estimates are derived for the
%semidiscrete mixed Galerkin approximation to (\ref{par}).
We use the usual notations for the $L^2,H_0^1$ and $H^2$ spaces and their norms
and semi-norms. Let $H^{-1}$ be the dual space of $H^1_0$ and let $\langle \cdot,\cdot\rangle$ be 
a duality paring between $H^{-1}$ and $H^1_0.$  Since we shall be dealing with time-space domain, we
further introduce for a Banach space $X$ with norm $\|\cdot\|_X,$ the space 
$L^p(0,T;X)$ denoted by $L^p(X),$ for $1\leq p \leq \infty$  with norm $\|\cdot\|_{L^p(X)}.$ Moreover,
denote by $H^m(X)$ the space of vector valued functions $\phi:(0,T)\longrightarrow X$ such that 
$\frac{d^j}{dt^j}\phi \in L^2(X)$ for $j=0,1,\ldots,m$ with the standard norm $\|\cdot\|_{H^m(X)}.$

For the  weak primal formulation, define a bilinear form  for $w,z\in H^1_0$
$$a(w,z):= ( A \nabla w, \nabla z).$$

Given $f\in L^2(L^2),$ $u_0\in H^1_0$ and $u_1\in L^2,$ the weak formulation of (\ref{par})-(\ref{ic-bc}) is to seek
a function $u: (0,T] \longrightarrow H^1_0$ with $u(0)=u_0$ and $u_t(0)=u_1$ such that 
\begin{eqnarray}
\langle u_{tt},w \rangle + a (u,w) = (f,w)~~~~\forall~w\in H^1_0. \label{weak}
\end{eqnarray}
Note that for  $f\in L^2(L^2),$ $u_0\in H^1_0$ and $u_1\in L^2,$ there exists a unique  weak  solution $u$
of (\ref{par})-(\ref{ic-bc}) satisfying $u\in L^2(H^1_0), u_t \in L^2(L^2)$ and $u_{tt} \in L^2(H^{-1}).$ Moreover,  
the equation (\ref{weak}) is satisfied for almost all $t\in (0,T]$. For a proof, refer to Evans 
(\cite{Evans}, pp. 399-408). 

For mixed formulation, 
let
$$ {\bf H}(div,\Omega)
=\{\bp\in (L^2(\Omega))^2:\nabla\cdot\bp\in L^2(\Omega) \} $$
be a Hilbert space equipped with norm $\|\bp\|_{{\bf H}(div,\Omega)}=(\|\bp\|^2+\|\nabla\cdot\bp\|
^2)^{\frac{1}{2}}.$

%For a mixed formulation,  
Now, introduce
\begin{equation}\label{bs}
 \bs=-A\nabla u,
\end{equation}
and 
$ \alpha=A^{-1}$.
Then, the equation (\ref{par}) is rewritten as
\begin{eqnarray}%-----------------------------mixed parabolic
 \alpha\bs+\nabla u =0, \quad u_{tt}+\nabla\cdot\bs =f.\label{mps}
\end{eqnarray}
Set $W=L^2(\Omega)$ and $\bV={\bf H}(div,\Omega).$  For given $f\in L^2(W),$ $u_0, u_1\in W,$
a weak mixed formulation for (\ref{par})-(\ref{ic-bc}) is to find
$(u,\bs): (0,T]\rightarrow W \times {\bV}$  with $u(0)=u_0$ and $u_t(0)=u_0$  such that $(u,\bs)$ satisfies
$$u,u_t\in L^2(W),\; u_{tt}\in L^2(H^{-1})\;\;\mbox{and}\; \bs \in L^2(\bV)$$
and 
\begin{eqnarray}
 ( \alpha\bs,\bv )-( u,\nabla\cdot\bv )=0~~~\forall ~\bv \in \bV,
\label{bsmf}\\
\langle u_{tt},w \rangle +(\nabla\cdot\bs,w) = (f,w)~~~~\forall~w\in W. \label{umf}
\end{eqnarray}

Below, we discuss the equivalence of weak primal and weak mixed formulations.
\begin{theorem}\label{Thm2.1}
The pair $(u,\bs)\in L^2(W)\times L^2(\bV)$  with $u_{tt} \in L^2(H^{-1})$ and $(u_0,u_1) \in W\times W$ 
is  a solution of the mixed formulation $(\ref{bsmf})$-$(\ref{umf})$ if and only if $u$ is a solution 
of the weak formulation $(\ref{weak})$  and $\bs = - A \nabla u$ with $u\in L^2(H^1_0)$ and $u_0\in H^1_0.$
%and let $(u_h,\bs_h)$ be a solution of the semidiscrete mixed formulation
% $(\ref{dbsmf})$-$(\ref{dumf})$.  Then the following a posteriori estimates hold
%for $\ell=0,1$:
\end{theorem}
\begin{proof}
Let $(u,\bs)\in L^2(W) \times L^2(\bV)$ with $u_{tt} \in L^2(H^{-1})$ be a solution of 
$(\ref{bsmf})$-$(\ref{umf})$  and let $\phi \in \mathcal{D}(\Omega).$
Choose $\bv = \mbox {Curl}\;\phi:=(-\partial \phi/\partial x_2, \partial \phi/\partial x_1 )$ in 
(\ref{bsmf}). Since it is divergence free, we obtain
$$  ( \alpha\bs, \mbox{Curl}\; \phi ) =0 \;\;\;{\mbox {a.e.}}\; t\in (0,T).$$
Using distributional derivative, it follows that 
$$  \langle {\mbox{Curl}}\;(\alpha\bs), \phi \rangle =0 \; \forall \phi \in \mathcal{D}(\Omega)
,\;{\mbox { a.e.}}\; t\in (0,T),$$
and hence, 
$$   {\mbox{Curl}}\;(A^{-1}\bs) =0 \; {\mbox{ in }}\;\mathcal{D}'(\Omega)
,\;{\mbox { a.e.}}\; t\in (0,T).$$
Now, a use of Helmholtz decomposition yields for some $\psi \in H^1_0$
$$  A^{-1}\bs = \nabla \psi\; \;{\mbox { a.e.}}\; t\in (0,T).$$
Apply this in (\ref{bsmf}) to arrive at
\begin{equation}\label{psi}
 (\nabla \cdot \bv, \psi + u)= 0 \; \forall \bv\in \bV,\;{\mbox { a.e.}}\; t\in (0,T),
 \end{equation}
which shows  $u=-\psi\in H^1_0.$ Hence for {\it a.e.} $t\in (0,T),$
$$A^{-1} \bs =- \nabla u$$
and 
\begin{equation*}\label{bs}
\bs = - A \nabla u \;{\mbox { a.e.}}\; t\in (0,T).
\end{equation*}
On substitution in (\ref{umf}) yields (\ref{par}) and hence, it satisfies (\ref{weak}) for {\it a.e.} 
$t\in (0,T).$ Since, $u,u_t \in L^2(W),$ $u\in C^0[0,T]$ and (\ref{psi}) holds for $t=0.$ 
Thus, $u_0=-\psi(0)\in H^1_0.$

For the converse, let $u$ be a weak solution of (\ref{par}) satisfying (\ref{weak}). Now, set 
$\bs = -A \nabla u \in L^2(L^2).$ Then, the equation (\ref{par}) becomes
$$ u_{tt} + \nabla\cdot \bs = f \;\;\mbox{ in }\; \mathcal{D}'(\Omega),\;{\mbox { a.e.}}\; t\in (0,T).$$
Since, $f\in L^2(L^2)$ and $\nabla\cdot \bs \in L^2(L^2),$ therefore, $u_{tt} +\nabla\cdot \bs \in L^2(L^2)$ and  (\ref{umf}) 
is satisfied
for {\it a.e.} $t\in (0,T).$ Note  that $\alpha \bs = -A^{-1} \bs = - \nabla u.$ Multiply by $\bv \in \bV,$ 
and integrate  over $\Omega$ to show that (\ref{bsmf}) is satisfied. This concludes the rest of the proof.
\end{proof}
Since the weak primal formulation (\ref{weak}) is well-posed, by the Theorem~\ref{Thm2.1} on equivalence, the weak mixed 
formulation is well-posed. As a byproduct, since  $f\in L^2(L^2)$ and $\nabla\cdot \bs \in L^2(L^2),$ 
therefore, $u_{tt} \in L^2(L^2).$

%Note that the boundary condition $u=0$ on $\partial \Omega$ is implicitly contained in (\ref{bsmf}).
Given $A$ uniformly positive definite, there exist two positive constants $a_0$ and $a_1$ such that
\begin{equation*}\label{alpha}
 a_0\|\bs\| \le \|\bs\|_{A^{-1}} \le a_1\|\bs\|,~~\mbox{where}~\|\bs\|_{A^{-1}
 }^2 :=(\alpha\bs,\bs).
\end{equation*}

\section {\bf A posteriori error estimates for the semidiscrete scheme} \label{sec:3}
%--------------------FOURTH
\se

This section focuses on a mixed finite element method for the hyperbolic problem
(\ref{par})-(\ref{ic-bc}) and  {\it a posteriori } error estimates are derived for the
semidiscrete mixed Galerkin approximation to (\ref{par})-(\ref{ic-bc}).

For the semidiscrete mixed formulation corresponding to \eqref{bsmf}-\eqref{umf}, let ${\mathcal T}_h = \{ K\}$ 
be a shape-regular partition of the domain $\Omega$ into triangles of diameter 
$h_K= \mbox{diam}(K)$. To each triangulation
${\mathcal T}_h$, we now associate a positive piecewise constant function $h(x)$ defined on ${\bar \Omega}$ by
$h|_K =h_K \; \forall K \in {\mathcal T}_h $. Let $\Gamma_h$ denote the set of all internal edges $E$ of the
triangulation ${\mathcal T}_h.$ 
%For simplicity of exposition of this section, we assume that the mesh
%size $h$ is kept a constant for all time $t \in (0,T]$. However, for completely discrete scheme, we shall allow $h$
%to vary from one time step to another, see Section 3 for details.
Further, let $\bV_h$ and $W_h$ be appropriate finite element subspaces of $\bV$ and $W$ satisfying 
LBB condition.
For more examples of these spaces including Raviart-Thomas-N\'ed\'elec finite element spaces, 
Brezzi- Douglas-Marini spaces and Brezzi-Douglas-Fortin-Marini spaces, see \cite{BF-91}.

The corresponding semidiscrete mixed finite element formulation is to seek a  
pair $(u_h,\bs_h):(0,T] \rightarrow W_h\times \bV_h$ such that
\begin{eqnarray}
 ( \alpha\bs_h,\bv_h )-( u_h, \nabla\cdot\bv_h )=0~~~\forall
 ~\bv_h \in \bV_h, \label{dbsmf}\\
 (u_{h,tt},w_h)+(\nabla\cdot\bs_h,w_h) = (f,w_h)~~~~\forall~w_h\in W_h
 \label{dumf}
\end{eqnarray}
with $u_h(0) \in W_h$ and $ u_{h,t}(0) \in W_h $ to be defined later. Since
$W_h$ and $\bV_h$ are finite dimensional, from (\ref{dbsmf}) we  can eliminate $\bs_h$ in the discrete level  
by writing it
in terms of $u_h.$ Therefore, substituting in (\ref{dumf}), we obtain a second order linear  system of ODEs and
existence follows using ODE linear theory. Then, as a consequence of LBB condition and energy
estimates, uniqueness can be proved easily  and hence, we skip the proof.

Set $e_u=u_h-u$ and $\es=\bs_h-\bs.$ From (\ref{bsmf})-(\ref{umf}) and 
(\ref{dbsmf})-(\ref{dumf}), $e_u$ and $\es$ satisfy the following
equations
\begin{eqnarray}
 (\alpha\es, \bv)-( e_u,\nabla\cdot\bv )=\br_1(\bv)~~~\forall ~\bv
 \in \bV, \label{ebsmf}\\
 (e_{u,tt}, w)+(\nabla\cdot\es, w)=\br_2(w) ~~~~\forall~w\in W,
 \label{eumf}
\end{eqnarray}
where the residuals $\br_1$ and $\br_2$ are given by
$$ \br_1(\bv):= (\alpha \bs_h,\bv )-( u_h,\nabla\cdot\bv ),$$
and
$$\br_2(w):= (u_{h,tt},w)+(\nabla\cdot\bs_h,w)-(f,w).$$

Following \cite{MNP-2012}, now introduce mixed elliptic reconstructions 
$\tilde u(t) \in H^1_0(\Omega)$ and $\tilde
\bs(t)\in \bV $ of $u_h(t)$ and $\bs_h(t)$ for $t\in (0,T]$, respectively, as follows:
for given $u_h$ and $\bs_h,$   let the mixed elliptic reconstructions
$\tilde u $ and $\tilde\bs $ satisfy
\begin{eqnarray}
 (\nabla\cdot(\tilde\bs-\bs_h),w)=-\br_2(w),~~~\forall w\in W,
 \label{ip1}\\
 (\alpha(\tilde\bs- \bs_h),\bv)-(\tilde u-u_h ,\nabla\cdot\bv )
= -\br_1(\bv),~~~\forall
\bv\in \bV. \label{ip2}
\end{eqnarray} 
Using Theorem 4.3 (pp. 132) of \cite{Braess}, one can verify  
that for a given $u_h,\bs_h,\br_1$ and $\br_2,$  the system (\ref{ip1})-(\ref{ip2}) has 
a unique pair of solution $\{ \tilde{u}(t),\tilde{\bs}(t)\} \in W \times \bV,$
for $t\in (0,T].$ Here elliptic reconstructions are assumed to be smooth in time.
%Now a simple modification of arguments in page 147 of \cite{Braess} provides more regularity for $\tilde{u},$ that is,
%$\tilde{u}(t)$ now belongs to $H^1_0(\Omega).$

Note that $\br_1(\bv_h)=0\;\;\forall \bv_h \in \bV_h,$ and $\br_2(w_h)=0 \;\;\forall
w_h \in W_h$.  Then, $\bs_h$ and $u_h$ are indeed mixed elliptic projections of
$\tilde\bs$ and $\tilde u,$ respectively. 
%For mixed elliptic  projections, we refer to \cite{VT}.

Using mixed elliptic reconstructions, we now rewrite 
$$e_u := (\tilde{u}-u)-(\tilde{u} -u_h)=: \xi_u-\eta_u,$$
and
$$\es :=  (\tilde\bs -\bs)-(\tilde\bs-\bs_h)=: \xi_{\bs}-\eta_{\bs}.$$
An application of (\ref{ip1})-(\ref{ip2}) in (\ref{ebsmf})-(\ref{eumf}) yields
\begin{eqnarray}
(\alpha\xi_{\bs}, \bv)-(\xi_u ,\nabla\cdot\bv ) &=& 0\;\; \forall \bv \in \bV,
 \label{xi1}\\
 (\xi_{u,tt}, w)+(\nabla\cdot\xi_{\bs}, w) &=& (\eta_{u,tt},w)\;\; \forall w\in W.
\label{xi0}
\end{eqnarray}
%\begin{lemma}%----------------properties of mixed elliptic reconstructions
With  mixed elliptic reconstructions $\tilde u $ and $\tilde\bs $ satisfying
 (\ref{ip1})-(\ref{ip2}), apply (\ref{xi1})  to check that
\begin{equation} \label{tilde-bs}
\alpha \tilde{\bs} = - \nabla \tilde{u}.
\end{equation}

\begin{lemma}\label{apriori}
 %------------------------------------------------- \hat\xi
Let $\xi_u$ and $ \xi_{\bs}$ satisfy $(\ref{xi1})$-$(\ref{xi0})$. Then, the
following estimates hold:
 \begin{equation}\label{esthx}%--------est. \hat\xi
  \|\xi_{u,t}(t)\| + \|\alpha^{1/2}\xi_{\bs}(t)\| \le
\|\xi_{u,t}(0)\| + \|\alpha^{1/2}\xi_{\bs}(0)\|+ 2  \int_{0}^{t}\|\eta_{u,tt}(s)\| \,ds.
 \end{equation}
and
 \begin{equation*}\label{estix}%--------est. int. \xi
 \|\xi_u(t)\| \leq \|\xi_u(0)\|  + 2\int_0^t\|\eta_{u,t}(s)\|\,ds.
%\|\xi_u(t)\| \leq \|\eta_{u,t}(0)\|  + \|\xi_{u,t}(0)\| + \|\eta_{u,t}\|.
 \end{equation*}
\end{lemma}
\begin{proof}
Differentiate (\ref{xi1}) with respect to $ t $ and set $ \bv=\xi_{\bs}$ in the 
resulting equation to find that
\begin{equation}\label{06}
(\alpha \xi_{\bs,t}, \xi_{\bs})- (\xi_{u,t}, \nabla\cdot\xi_{\bs}) = 0.
\end{equation}
%We shall first prove the result \eqref{esthx} when $l=1$.
 Choose $w =\xi_{u,t} $ in (\ref{xi0}). Then, add the resulting
equations to (\ref{06}) to arrive at
\begin{eqnarray}
 \frac{1}{2}\frac{d}{dt}(\|\xi_{u,t}\|^2+ \|\alpha^{1/2}\xi_{\bs}\|^2 ) &=&
 (\eta_{u,tt}, \xiut).\label{07}
 %& \leq  & \|\eta_{u,tt}\| \|\xiut\| .\label{07}
\end{eqnarray}
On integrating (\ref{07}) from $0$ to $t$, a use of the Cauchy-Schwarz inequality yields
\begin{equation*}\label{xi-u-bs}
\|\xi_{u,t}(t)\|^2 + \|\alpha^{1/2} \xi_{\bs}\|^2  \le \|\xi_{u,t}(0) \|^2 +\|\alpha^{1/2} \xibs(0)\|^2 + 2 \int_{0}^t \|\eta_{u,tt}(s)\| \,\|\xiut(s)\| ds.
\end{equation*}
Setting
$$
|||(\xiut,\xibs)(t)|||=  (\|\alpha^{1/2}\xibs(t)\|^2 + \|\xiut(t)\|^2)^{1/2},
$$
let $ t^* \in [0,t] $ be such that
$$
|||(\xiut,\xibs)(t^*)||| =\displaystyle \max_{0 \leq s \leq t} |||(\xiut,\xibs)(s)|||.
$$
Then at time $ t=t^* $, equation (\ref{xi-u-bs}) becomes
\begin{equation*}
|||(\xiut,\xibs)(t^*)|||  \leq |||(\xiut,\xibs)(0)||| + 2 \int_0^{t^\ast} \|\eta_{u,tt}(s)\|\,ds,
\end{equation*}
and hence,
\begin{equation}\label{eq:xiut}
|||(\xiut,\xibs)(t)|||  \leq |||(\xiut,\xibs)(0)||| + 2 \int_0^{t^\ast} \|\eta_{u,tt}(s)\|\,ds.
\end{equation}
This completes the proof of (\ref{esthx}). Note that from (\ref{eq:xiut}), we obtain $L^{\infty}(L^2)$-estimate
 of the displacement using $\xi(t) = \xi(0) + \int_{0}^t \xiut(s)\;ds.$ Now in order to reduce the regularity,
 %To prove (\ref{estix}), we
an integration tool which is a variant  of Baker's time testing procedure is used in a crucial way. 
To motivate our tool, integrate (\ref{xi0}) with respect to time to arrive at
\begin{equation}\label{xi-ut-bs}
(\xiut,w)+(\nabla \cdot \hat{\xi}_{\bs},w) = (\xi_{u,t}(0),w) + (\eta_{u,t} ,w) - (\eta_{u,t}(0),w),
\end{equation}
where $ \hat{\xi}_{\bs} = \int_0^t \xibs(s)~ds. $ Choose $ w = \xiu $ in (\ref{xi-ut-bs}) and $ \bv = \hat{\xi}_{\bs} $ in (\ref{xi1}) and adding the resulting equations to obtain
\begin{eqnarray*}
\dfrac{1}{2} \dfrac{d}{dt} \Big (  \|\xiu(t)\|^2 + 
\|\alpha^{1/2}\hat{\xi}_{\bs}(t)\|^2 \Big )  &=& 
( e_{u,t}(0), \xi_u) + (\eta_{u,t} , \xi_u).
\end{eqnarray*}
Then,  integrate with respect to time and use kick back arguments to arrive at
\begin{equation*}
\|\xiu(t)\| \leq  \|\xi_u(0)\| + 2\int_0^t\|\eta_{u,t}(s)\|\,ds.
\end{equation*}
This completes the rest of the proof. % proof of \eqref{estix}.
\end{proof}

Assume that there exists a linear operator $\Pi_h:\bV\rightarrow \bV_h$ such that $\nabla
\cdot\Pi_h= P_h(\nabla\cdot~),$ \\
where $P_h:~W\rightarrow W_h$ is the $L^2$-projection defined by
$$(\phi-P_h\phi,w_h)=0~~\forall~w_h\in W_h,~\phi\in W.$$
Further, we assume that the finite element spaces satisfy the following
properties:
$$
\|\bv-\Pi_h\bv\| \le C_I h^{r}\|\bv\|_{r},\;\; 1 \leq r \leq \ell+1, ~~~~\|w-P_h w\|
\le C_I h^{r}\|w\|_{r},\;\;\; 0\leq r \leq \ell+1.
$$
\noindent Note that for $\bv \in \bH(div,\Omega)$ and $ w \in L^2(\Omega),$
the following properties hold true:
\begin{equation*}\label{aprx1}
 (\nabla\cdot(\bv-\Pi_h\bv),w_h)=0,~~w_h\in W_h;~~~(w-P_h w,\nabla\cdot\bv_h)=0,
 ~~\bv_h \in \bV_h.
\end{equation*}
Examples of spaces satisfying the above can be found in \cite {BF-91}.

To prove the main theorem of this section, we need the following  {\it a posteriori} estimates
of $\eta_{u}, \eta_{u,t}$ and $\eta_{\bs}$ related to the mixed
elliptic reconstructions (\ref{ip1})-(\ref{ip2}). For a proof, see \cite{CC-97}.
\begin{lemma}\label{negative}
For Raviart-Thomas-N\'ed\'elec elements, there exists a positive constant
$C$ which depends only on the coefficient matrix $A,$ the domain $\Omega,$
the shape regularity of the elements and polynomial degree $ \ell $ such that for
$ \ell =0,1,$
\begin{equation}\label{mer-u}
\|\eta_u\|  \leq  C \Big( \|h^{\ell+1} \br_2 \| + \min_{w_h\in W_h}
\|h (\alpha \bs_h -\nabla_h w_h)\| \Big),
\end{equation}
%\nonumber \\
%&\leq& C \Big( \|h^{ k+1 } (u_{h,t}-f+ \nabla\cdot \bs_h) \|^2
%+ \min_{w_h\in W_h}\|h (\alpha \bs_h -\nabla_h w_h )\|^2 \Big),\\
and for $j=1,2$,
\begin{equation} \label{mer-ut}
\left\|\frac{\partial^j\eta_{u}}{\partial t^j}\right\|  \leq  C
\left( \left\|h^{\ell+1} \frac{\partial^j\br_{2}}{\partial t^j}\right \|
+ \min_{w_h\in W_h}\left\|h \left(\alpha \frac{\partial^j\bs_h}{\partial t^j} -
\nabla_h w_h \right)\right\| \right),
\end{equation}
%\nonumber \\
%&\leq & C \Big( \|h^{k+1} (u_{h,t}-f+ \nabla\cdot \bs_h)_t \|^2
%+ \min_{w_h\in W_h}\|h (\alpha \bs_{h,t} -\nabla_h w_h)\|^2 \Big),
and
\begin{equation} \label{mer-bs}
\|\alpha^{1/2}\eta_{\bs}\| \leq
C \left( \|h \br_2 \| + \| h^{1/2} J(\alpha \bs_h\cdot
{\bf t})\|_{0,\Gamma_h} + \|h \,{\mbox {curl}}_h\,(\alpha \bs_h)\| \right),
%\nonumber  \\
%&\leq& C \Big( \|h (u_{h,t}-f+ \nabla\cdot \bs_h) \|^2
%+ \| h^{1/2} J(\alpha \bs_h\cdot {\bf t})\|_{0,\Gamma_h}^2
%+ \|h \,{\mbox {curl}}_h\,(\alpha \bs_h)\|^2 \Big), \label{mer-bs}
\end{equation}
\end{lemma}
\noindent
where $\br_2= (u_{h,tt}-f+ \nabla\cdot \bs_h) $ is a residual and  $J(\alpha \bs_h \cdot {\bf t})$ denotes the jump of
$\;\alpha \bs_h \cdot {\bf t}\; $ across element edge ${{E}}$ with ${\bf t}$ being
the tangential unit vector along the edge ${{E}} \in \Gamma_h.$

Now, let ${\mathcal  E}_1(\br_2,\bs_h;{\mathcal  T}_h)$,
${\mathcal  E}_1(\frac{\partial^j\br_{2}}{\partial t^j},\frac{\partial^j\bs_h}{\partial t^j};{\mathcal  T}_h)$
and ${\mathcal  E}_2(\br_2,\bs_h;{\mathcal  T}_h)$ denote the terms on the right-hand sides of
(\ref{mer-u}), (\ref{mer-ut}) and (\ref{mer-bs}), respectively. Then,
using Lemmas \ref{apriori}-\ref{negative}, we finally obtain the main theorem of this section as:
%%%%%%%%%%%%%%%%%%%%%%%%%%%%%%%%%%%%%%%%%%%%%%%%%%%%%%%%%%%%%%%%%%%%%%%%%%%%%%%%%%%%%%%%%%%%%%%%%%
\begin{theorem}
Let $(u,\bs)$ be a solution of the mixed formulation $(\ref{bsmf})$-$(\ref{umf})$
and let $(u_h,\bs_h)$ be a solution of the semidiscrete mixed formulation
 $(\ref{dbsmf})$-$(\ref{dumf})$.  Then the following a posteriori estimates hold
for $\ell=0,1$:
\begin{eqnarray*} \label{error-u}
\|e_{u,t}\|_{L^\infty(0,T;L^2(\Omega))}&+& \|\alpha^{1/2}e_{\bs}
\|_{L^\infty(0,T;L^2(\Omega))}\nonumber\\
&\lesssim&
\|e_{u,t}(0)\| + \|\alpha^{1/2}e_{\bs}(0)\|
+{\mathcal  E}_1(\br_{2,t}(0),\bs_{h,t}(0);{\mathcal  T}_h)\nonumber\\
&&+{\mathcal  E}_2(\br_2(0),\bs_h(0);{\mathcal  T}_h)+\|{\mathcal  E}_1(\br_{2,t},\bs_{h,t};
{\mathcal  T}_h)\|_{L^\infty(0,T)}\nonumber\\
&&+\|{\mathcal  E}_2(\br_{2},\bs_{h};{\mathcal  T}_h)\|_{L^\infty(0,T)}
 + \int_0^T{\mathcal  E}_1(\br_{2,tt},\bs_{h,tt};{\mathcal  T}_h)\,ds,
 \end{eqnarray*}
%\begin{eqnarray} \label{error-u}
%\|e_{u,t}\|_{L^\infty(0,T;L^2(\Omega))}+ \|\alpha^{1/2}e_{\bs}
%\|_{L^\infty(0,T;L^2(\Omega))}
%&\lesssim&
%\|e_{u,t}(0)\| + {\mathcal  E}_1(\br_{2,t}(0),\bs_{h,t}(0),{\mathcal  T}_h)\|\alpha^{1/2}e_{\bs}(0)\|\nonumber\\
%&&+\|\alpha^{1/2}e_{\bs}(0)\|
%+{\mathcal  E}_2(\br_2(0),\bs_h(0),{\mathcal  T}_h)\nonumber\\
%&&+\|{\mathcal  E}_1(\br_{2,t},\bs_{h,t},
%{\mathcal  T}_h)\|_{L^\infty(0,T)}\nonumber\\
%&&+\|{\mathcal  E}_2(\br_{2},\bs_{h},{\mathcal  T}_h)\|_{L^\infty(0,T)}
% + \int_0^T{\mathcal  E}_1(\br_{2,tt},\bs_{h,tt},{\mathcal  T}_h)\,ds,
% \end{eqnarray}
and
\begin{eqnarray*} \label{error-bsL2}
\|e_u\|_{L^\infty(0,T;L^2(\Omega))} &\lesssim&
\|e_{u}(0)\|+{\mathcal  E}_1(\br_2(0),\bs_h(0);{\mathcal  T}_h)\nonumber\\
&&+\|{\mathcal  E}_1(\br_{2},\bs_{h};{\mathcal  T}_h)\|_{L^\infty(0,T)}
+\int_0^T{\mathcal  E}_1(\br_{2,t},\bs_{h,t};{\mathcal  T}_h)\,ds.
\end{eqnarray*}
%where the estimates of $\eta_{u}, \eta_{u,t}$ and $\eta_{\bs}$ are given respectively by (\ref{mer-u}), %(\ref{mer-ut}) and (\ref{mer-bs}).
%where $C_1$ and $C_2(a_1,T)$ are defined in \eqref{c1c2}.
\end{theorem}

%%%%%%%%%%%%%%%%%%%%%%%%%%%%%%%%%%%%%%%%%%%%%%%%%%%%%%%%%%%%%%%%%%%%%%%%%%%%%%%%%%%%%%%%%%%%%%%%%%
%%%%%%%%%%%%%%%%%%%%%%completely discrete scheme%%%%%%%%%%%%%%%%%%%%%%%%%%%%%%%%%%%%%%%%%%%%%%%%%%
%%%%%%%%%%%%%%%%%%%%%%%%%%%%%%%%%%%%%%%%%%%%%%%%%%%%%%%%%%%%%%%%%%%%%%%%%%%%%%%%%%%%%%%%%%%%%%%%%%
\section {\bf Completely  discrete scheme} \label{sec:4}
%--------------------FOURTH
\se
This section deals with {\it a posteriori} analysis for a completely
discrete mixed approximation based on  backward differencing.

Let $0=t_0< t_1 < \ldots < t_N=T,$ $I_n = (t_{n-1}, t_n]$ and
$k_n =t_{n}-t_{n-1}.$ For  $n \in [0:N],$ let ${\mathcal {T}}_n$ be a refinement of
a macro-triangulation which is a triangulation of the domain $\Omega $  that
satisfies the same conformity and shape
regularity assumptions  made on its refinements. Let
$$h_n(x):= {\mbox {diam} }\;(K),\;\;{\mbox{ where }}
\; K \in {\mathcal  { T }}_n \;{\mbox {and }}\; x \in K, $$
for all $x\in \Omega.$ Given two compatible triangulations ${\mathcal  {T}}_{n-1}$
and $ {\mathcal  { T }}_n,$ i.e., they are refinements of the same
macro-triangulation, let $\hat{\mathcal  { T}}_n$ be the finest common
coarsening of  $ {\mathcal  { T }}_n$ and  ${\mathcal  {T}}_{n-1},$ whose meshsize is
given by $\hat{h}_n := \max ( h_n, h_{n-1}),$  see (\cite{LM-06}, pp. 1655).

We consider $ \bV^n_h $ and $ W^n_h $
defined over the triangulations  ${\mathcal {T}}^n $ as Raviart-Thomas
finite element spaces of index $\ell \geq 0$ of $ \bH(div,\Omega)$
and  $ L^2(\Omega),$ respectively. Let $P_h^n: L^2(\Omega)\longrightarrow W^n_h $ be the
$ L^2 $-projection defined by
$$ (P_h^n w, \phi^n) = ( w, \phi^n)\;\;\; \forall \phi^n \in W_h^n.$$
Given $U^0=P_h^0 u_0$, find $\{(U^n, \bS^n)\}$ with $ (U^n, \bS^n)
\in W_h^n \times \bV_h^n $ for $n \in [1:N]$ such that
\begin{eqnarray} \label{bS-U0}
(\partial_t^2 U^n, w) +(\nabla \cdot \bS^n, w)&=& (\bar{f}^n, w)\;\;
\forall w \in W_h^n,\\
(\alpha \bS^n, \bv) - (U^n, \nabla\cdot \bv)&=& 0 \;\;\; \forall \bv \in \bV_h^n,
\label{bS-U1}
\end{eqnarray}
%where $\bar f^0(\cdot) :=f(0,\cdot)$ and  $\bar{f}^n(\cdot) :=k_n^{-1}\int_{t_{n-1}}^{t_n}f(t,\cdot)\,dt$ for $n\geq %1$, and the backward second and first finite differences are given, respectively, by
where $\bar{f}^n$ is either chosen as point-wise value  $\bar f^n(\cdot) :=f(t_n,\cdot)$ for $n\geq 0$ or the one through average value as in Remark \ref{remark4.2}. Here,  the backward second and first finite differences are given, respectively, by
\begin{equation*}
\partial_t^2 U^n=\frac{\partial_t U^n-\partial_t U^{n-1}}{k_n},
\end{equation*}
and
$$\partial_t U^n :=\left\{\begin{array}{ll}\displaystyle
\frac{U^n- U^{n-1}}{k_n}, & \mbox{ for } n=1,\cdots, N,\\
P_h^0u_1, & \mbox{ for } n=0.\end{array}\right.
$$
Throughout the rest of the paper, we shall use the following notation:
$$
P_h^n(\partial_t \phi^n)
=\frac{1}{k_n}(\phi^n-P_h^n \phi^{n-1}).
$$
Following \cite{GLM-2013},  we define for given a sequence of discrete values $\{V^n\}_{n=0}^N,$
the time reconstruction $V:[0,T]\times\Omega\rightarrow \R$ or $\R^2$ as
\begin{equation}\label{U-interpolant}
%V(t)  &=&   \frac{(t-t_{n-1})}{k_n} V^{n}+ \frac{(t_n-t)}{k_n} V^{n-1}
%-\frac{(t-t_{n-1})(t_n-t)^2}{k_n}\partial_t^2V^n,
%\;\;\; t_{n-1} < t \leq t_n,\nonumber\\
V(t)= V^{n}+ (t-t_n) \partial_tV^{n}
-\frac{(t-t_{n-1})(t_n-t)^2}{k_n}\partial_t^2V^n,
\;\;\; t_{n-1} < t \leq t_n,\\
\end{equation}
for $n=1,\cdots, N$. Note that we have used the fact that $\partial_tV^0$ is well defined.

We shall use the above $C^1$-function $V(t)$ such that for $n=0,1,\cdots, N$,
\begin{equation}\label{C1}
V(t_n)=V^n,\qquad V_t(t_n)=\partial_tV^n,\qquad V_{tt}(t)=(1+\mu^n)\partial_t^2V^n,
\end{equation}
for $t\in(t_{n-1},t_n]$, where 
$$\mu^n(t):=-6k_n^{-1}(t-t_{n-1/2}).$$

Similarly, we define $C^1$-functions $U(t)$ and $\bS(t)$ in time variable using the discrete sequences
$\{U^n\}_{n=0}^N$ and $\{\bS^n\}_{n=0}^N$, respectively.

As in Section 3, for given $\{U^n,\bS^n\}_{n=0}^N$, we now define the
 mixed elliptic reconstructions  $\tilde u^n \in H^1_0(\Omega) $ and $\tilde\bs^n \in \bV $ at
$t=t_n$ as:
\begin{eqnarray}
 (\nabla\cdot(\tilde\bs^n-\bS^n),w)=-\br_2^n(w),~~~w\in W,
 \label{mer1}\\
 (\alpha(\tilde\bs^n- \bS^n),\bv)-(\tilde u^n-U^n ,\nabla\cdot\bv)
= -\br_1^n(\bv),~~~ \bv\in \bV, \label{mer2}
\end{eqnarray}
where $\br_1^n(\bv):=(\alpha\bS^n,\bv)-(U^n,\nabla\cdot\bv)$ and 
$\br_2^n(w):=(P^n_h(\partial_t^2U^n),w)+(\nabla\cdot\bS^n,w)-(\bar{f}^n,w).$

Since $\br_1^n(\bv_h)=0\;\;\forall \bv_h \in \bV_h^n,$ $n\geq 0$ and $\br_2^n(w_h)=0 \;\;
\forall w_h \in W_h^n,$ $n\geq 1$, in fact, $\bS^n$ and $U^n$ are mixed elliptic
projections of $\tilde\bs^n $ and $\tilde u^n $ at time $t=t_n,$
respectively. Now given $\{\tilde{u}^n\}_{n=0}^N$ and $\{\tilde{\bs}^n\}_{n=0}^N$, we define
the $C^1$-functions $\tilde{u}(t)$ and $\tilde{\bs}(t)$ in time $t\in (0,T]$, respectively, as
\begin{equation}\label{U-interpolant1}
\tilde{u}(t)=\tilde{u}^{n}+ (t-t_{n})\partial_t\tilde{u}^{n}
-\frac{(t-t_{n-1})(t_n-t)^2}{k_n}\partial_t^2\tilde{u}^n,
\;\;\; t_{n-1} < t \leq t_n,
\end{equation}
and
\begin{equation}\label{bS-interpolant1}
\tilde{\bs}(t)= \tilde{\bs}^{n}+ (t-t_n)\partial\tilde{\bs}^{n}
-\frac{(t-t_{n-1})(t_n-t)^2}{k_n}\partial_t^2\tilde{\bs}^n,
\;\;\; t_{n-1} < t \leq t_n,
\end{equation}
provided that $\partial_t\tilde{u}^0$ and $\partial_t\tilde{\bs}^0$ are well defined.

For $t\in (0,T]$, the mixed elliptic reconstruction $\{\tilde{u},\tilde{\bs}\}$ satisfies
\begin{eqnarray}
 (\nabla\cdot(\tilde\bs-\bS),w)=-\br_2 (w),~~~w\in W,
 \label{mer3}\\
 (\alpha(\tilde\bs- \bS),\bv)-(\tilde u-U ,\nabla\cdot\bv)
= -\br_1 (\bv),~~~
\bv\in \bV, \label{mer4}
\end{eqnarray}
where  $ \br_1 $ and $ \br_2 $ are defined as $C^1$-functions in time using $\{\br_1^n,\br_2^n\}_{n=1}^{N}$ as in (\ref{U-interpolant}).

Again, set
\begin{equation*}\label{eu}
e_u = (\tilde{u}-u)-(\tilde{u} -U)=: \xi_u-\eta_u,
\end{equation*}
and
\begin{equation*}\label{es}
\es =  (\tilde\bs -\bs)-(\tilde\bs-\bS)=: \xi_{\bs}-\eta_{\bs}.
\end{equation*}
%Using (\ref{mer1})-(\ref{mer2}) in (\ref{error-1})-(\ref{error-2}), we obtain
%
%
Now, the pair $\{e_u,\es\}$ satisfies
\begin{equation}\label{3.14}
(e_{u,tt},w)+(\nabla\cdot\es,w) = (U_{tt},w)+(\nabla\cdot\bS,w)-(f,w).
\end{equation}
On splitting $e_u$ and $\es$, we obtain from (\ref{3.14})
\begin{eqnarray}\label{error-3}
(\xi_{u,tt},w)+(\nabla\cdot \xi_{\bs},w) &=& (\eta_{u,tt},w)+((I-P^n_h)U_{tt},w)+
\mu^n(t)(\partial_t^2U^n,P^n_hw)\\
&&+(\nabla\cdot (\tilde{\bs}-\tilde{\bs}^n), w) + (\bar{f}^n-f,w)\;\;\;\forall w\in W. \nonumber
\end{eqnarray}
Similarly, we also arrive at
\begin{eqnarray}\label{error-4}
(\alpha \xi_{\bs},\bv)-(\xi_u, \nabla\cdot \bv)&=&
( \alpha (\tilde{\bs}-\tilde{\bs}^n),\bv)
- (\tilde{u}-\tilde{u}^n,\nabla\cdot \bv)\;\;\; \forall \bv\in \bV.
\end{eqnarray}
Note that
\begin{equation}\label{3.17}
\tilde{\bs}-\tilde{\bs}^n= (t-t_n)  \partial_t\tilde{\bs}^n
+\left(k_n^{-1}(t_n-t)^3-(t_n-t)^2\right)\partial_t^2\tilde{\bs}^n,
\end{equation}
and
\begin{equation}\label{3.18}
\tilde{u}-\tilde{u}^n= (t-t_n)   \partial_t\tilde{u}^n
+\left(k_n^{-1}(t_n-t)^3-(t_n-t)^2\right)\partial_t^2\tilde{u}^n.
\end{equation}
Now, it follows that
\begin{eqnarray}\label{3.19}
( \alpha (\tilde{\bs}&-&\tilde{\bs}^n),\bv)
- (\tilde{u}-\tilde{u}^n,\nabla\cdot \bv)= (t-t_n) 
\Big\{( \alpha \partial_t \tilde{\bs}^n, \bv)
 - ( \partial_t \tilde{u}^n , \nabla  \cdot \bv)\Big\}\\
 &&+\left(k_n^{-1}(t_n-t)-(t_n-t)^2\right)
 \Big\{( \alpha \partial_t^2 \tilde{\bs}^n, \bv)
 - ( \partial_t^2 \tilde{u}^n , \nabla  \cdot \bv)\Big\},\nonumber
\end{eqnarray}
and from (\ref{mer1}) with definition of $\br_1(\bv)$, we find that
%From the definition of  mixed elliptic reconstructions
%(\ref{mer1})-(\ref{mer2}), we note that at time level $t=t_n$
\begin{equation}\label{3.20}
( \alpha\tilde{\bs}^n,\bv)-(\tilde{u}^n, \nabla\cdot \bv) = 0 \;\; \forall \bv \in \bV.
\end{equation}
From (\ref{3.20}), the equation (\ref{3.19}) takes the form
\begin{equation*}\label{mer-property}
(\alpha(\tilde{\bs}-\tilde{\bs}^n),\bv)-(\tilde{u}-\tilde{u}^n, \nabla\cdot \bv) =0,
\end{equation*}
and thus,  (\ref{error-4}) becomes
\begin{equation}\label{error-4i}
(\alpha \xi_{\bs},\bv)-(\xi_u, \nabla\cdot \bv)= 0 \;\;\; \forall \bv\in \bV.
\end{equation}
%From (\ref{error-4i}), we easily find that
%\begin{equation}\label{xibs-xiu}
%$$
%\alpha \xi_{\bs}= -\nabla \xi_u,
%$$
%\end{equation}
%and hence, we note that the mixed elliptic reconstructions satisfy
%$$
%\alpha \tilde{\bs}= -\nabla \tilde{u}.
%$$

\begin{theorem}\label{xius-thm}
Let $(u,\bs)$ and $(U,\bS)$ be the solution of $(\ref{bsmf})$-$(\ref{umf})$  and 
$(\ref{bS-U0})$-$(\ref{bS-U1}),$ respectively. Then, the
following estimates hold for $t\in (t_{n-1},t_n]$
$$
\|\xi_{u,t}(t)\|+\|\alpha^{1/2}\xi_{\bs}(t)\|\leq\|\xi_{u,t}(0)\|+\|\alpha^{1/2}\xi_{\bs}(0)\|
+2\sum_{j=1}^4{\mathcal  E}_{1,j}(t)+2\int_0^t\|\eta_{u,tt}(s)\|\,ds,
$$
where 
\begin{eqnarray*}
{\mathcal  E}_{1,1}(t)&=&\sum_{j=1}^{n-1}\int_{t_{j-1}}^{t_j}||(I-P_h^j)U_{tt}||\,ds
+\int_{t_{n-1}}^{t}||(I-P_h^n)U_{tt}||\,ds,\nonumber\\
{\mathcal  E}_{1,2}(t)&=& \sum_{j=1}^{n-1}\int_{t_{j-1}}^{t_j}||\mu^j\partial_t^2 U^j||\,ds
+\int_{t_{n-1}}^{t}||\mu^n\partial_t^2 U^n||\,ds,\nonumber\\
{\mathcal  E}_{1,3}(t)&=&\sum_{j=1}^{n-1}
\left(\frac{k_j^2}{2}||\partial_t(\br_2^j-\nabla\cdot\bS^j)||+
\frac{k_j^3}{12}||\partial_t^2(\br_2^j-\nabla\cdot\bS^j)||\right)\nonumber\\
&&+\int_{t_{n-1}}^{t}(t_n-s)||\partial_t(\br_2^n-\nabla\cdot\bS^n)||\,ds\nonumber\\
&&+\int_{t_{n-1}}^{t}\left((t_n-s)^2-\frac{(t_n-s)^3}{k_n}\right)
||\partial_t^2(\br_2^n-\nabla\cdot\bS^n)||\,ds,\\
{\mathcal  E}_{1,4}(t)&=&\sum_{j=1}^{n-1}\int_{t_{j-1}}^{t_j}||\bar{f}^j-f||\,ds+
\int_{t_{n-1}}^{t}||\bar{f}^n-f||\,ds.
\end{eqnarray*}
\end{theorem}
\noindent
{\it Proof }. Differentiate (\ref{error-4i}) with respect to $t$. Then, choose $v=\xi_{\bs}$ in the 
resulting equation and $w=\xi_{u,t}$ in (\ref{error-3}) to obtain for $t\in(t_{n-1},t_n]$
\begin{eqnarray*}\label{3.25}
\frac{1}{2}\left(\|\xi_{u,t}\|^2+\|\alpha^{1/2}\xi_{\bs}\|^2\right)
&=&
(\eta_{u,tt},\xi_{u,t})+((I-P^n_h)U_{tt},\xi_{u,t})\\
&&+
\mu^n(t)(\partial_t^2U^n,P^n_h\xi_{u,t})\\
&&+(\nabla\cdot (\tilde{\bs}-\tilde{\bs}^n), \xi_{u,t}) + (\bar{f}^n-f,\xi_{u,t}).\nonumber
\end{eqnarray*}
On integrating from $0$ to $t$ with $t\in(t_{n-1},t_n]$, we find that
\begin{eqnarray}\label{3.26}
\frac{1}{2}\left(\|\xi_{u,t}\|^2+\|\alpha^{1/2}\xi_{\bs}\|^2\right)
&=&\frac{1}{2}\left(\|\xi_{u,t}(0)\|^2+\|\alpha^{1/2}\xi_{\bs}(0)\|^2\right)
+\int_0^t(\eta_{u,tt},\xi_{u,t})\,ds\\
&&+J_{1,1}^n(\xi_{u,t})+J_{1,2}^n(\xi_{u,t})+J_{1,3}^n(\xi_{u,t})+J_{1,4}^n(\xi_{u,t}),\nonumber
\end{eqnarray}
where
\begin{eqnarray*}
J_{1,1}^n(\xi_{u,t})&:=&\sum_{j=1}^{n-1}\int_{t_{j-1}}^{t_j}\left((I-P_h^j)U_{tt},\xi_{u,t}\right)ds
+\int_{t_{n-1}}^{t}\left((I-P_h^n)U_{tt},\xi_{u,t}\right)ds,\\
J_{1,2}^n(\xi_{u,t})&:=&\sum_{j=1}^{n-1}\int_{t_{j-1}}^{t_j}\mu^j(\partial_t^2U^j,P_h^j\xi_{u,t})\,ds
+\int_{t_{n-1}}^{t}\mu^n(\partial_t^2U^n,P_h^n\xi_{u,t})\,ds,\\
J_{1,3}^n(\xi_{u,t})&:=&\sum_{j=1}^{n-1}\int_{t_{j-1}}^{t_j}(\nabla\cdot(\tilde{\bs}-\tilde{\bs}^j),\xi_{u,t})\,ds
+\int_{t_{n-1}}^{t}(\nabla\cdot(\tilde{\bs}-\tilde{\bs}^n),\xi_{u,t})\,ds,
\end{eqnarray*}
and
$$
J_{1,4}^n(\xi_{u,t}):=\sum_{j=1}^{n-1}\int_{t_{j-1}}^{t_j}(\bar{f}^j-f,\xi_{u,t})\,ds
+\int_{t_{n-1}}^{t}(\bar{f}^n-f,\xi_{u,t})\,ds.
$$
Set
$$
E_1^2(t):=\|\xiut(t)\|^2+\|\alpha^{1/2}\xibs(t)\|^2,
$$
and let at $t=t^* \in (0,t]$ be such that
$$
E_1(t^*) =\displaystyle \max_{0 \leq s \leq t} E_1(s).
$$
Now, a use of the Cauchy-Schwarz inequality yields
\begin{equation*}\label{3.27}
|J_{1,1}^n(\xi_{u,t})|\leq\left(\sum_{j=1}^{n-1}\int_{t_{j-1}}^{t_j}\|(I-P_h^j)U_{tt}\|\,ds
+\int_{t_{n-1}}^{t}\|(I-P_h^n)U_{tt}\|\,ds\right)E_1(t^*).
\end{equation*}
and  similarly,
$$
|J_{1,2}^n(\xi_{u,t})|\leq\left(\sum_{j=1}^{n-1}\int_{t_{j-1}}^{t_j}||\mu^j\partial_t^2 U^j||\,ds
+\int_{t_{n-1}}^{t}||\mu^n\partial_t^2 U^n||\,ds\right)
E_1(t^*).
$$
For $J_{1,3}^n$, we rewrite using (\ref{3.17}) as
\begin{eqnarray}\label{3.28}
\left(\nabla\cdot(\tilde{\bs}-\tilde{\bs}^n),w\right)&=& (t-t_n)  
\left(\nabla\cdot\partial_t\tilde{\bs}^n,w\right)\\
&&+\left(k_n^{-1}(t_n-t)^3-(t_n-t)^2\right)\left(\nabla\cdot\partial_t^2\tilde{\bs}^n,w\right)\nonumber.
\end{eqnarray}
From (3.6), we obtain 
\begin{equation}\label{3.29}
\left(\nabla\cdot\tilde{\bs}^n,w\right)=-\br_2^n(w)+\left(\nabla\cdot\bS^n,w\right),
\end{equation}
and therefore, for $j=1,2$
\begin{equation}\label{3.30}
\left(\nabla\cdot\partial_t^j\tilde{\bs}^n,w\right)=-(\partial_t^j\br_2^n)(w)+
\left(\nabla\cdot\partial_t^j\bS^n,w\right).
\end{equation}
On substituting (\ref{3.30}) for $j=1,2$ in $J^n_{1,3}$, we  arrive at
\begin{eqnarray*}
J^n_{1,3}(\xiut)&=&\sum_{j=1}^{n-1}\int_{t_{j-1}}^{t_j}\Big((t_j-s)\left\{(\partial_t\br_2^j)(\xiut)-
\left(\nabla\cdot\partial_t\bS^j,\xiut\right)\right\}\\
&&\qquad -\left(k_j^{-1}(t_j-s)^3-(t_j-s)^2\right)\left\{(\partial_t^2\br_2^j)(\xiut)-
\left(\nabla\cdot\partial_t^2\bS^j,\xiut\right)\right\}\Big)\\
&+&\int_{t_{n-1}}^{t}\Big( (t_n-s)\left\{(\partial_t\br_2^n)(\xiut)-
\left(\nabla\cdot\partial_t\bS^n,\xiut\right)\right\}\\
&&\qquad-\left(k_n^{-1}(t_n-s)^3-(t_n-s)^2\right)\left\{(\partial_t^2\br_2^n)(\xiut)-
\left(\nabla\cdot\partial_t^2\bS^n,\xiut\right)\right\}\Big).
\end{eqnarray*}
Using the Cauchy-Schwarz inequality, it follows that
\begin{eqnarray*}
\left|J^n_{1,3}(\xiut)\right|&\leq&\sum_{j=1}^{n-1}\left(
\frac{k_j^2}{2}\left\|\partial_t(\br_2^j-\nabla\cdot\bS^j)\right\|+
\frac{k_j^3}{12}\left\|\partial_t^2(\br_2^j-\nabla\cdot\bS^j)\right\|\right)E_1(t^*)\\
&&+\Big(\int_{t_{n-1}}^{t}\big\{(t_n-s)
\|\partial_t(\br_2^n-\nabla\cdot\bS^n)\| \\
&&\quad +\left((t_n-s)^2-k_n^{-1}(t_n-s)^3\right)\|\partial_t^2(\br_2^n-
\nabla\cdot\bS^n)\|\big\}\Big)E_1(t^*). 
\end{eqnarray*}
For $J^n_{1,4}$, we note that 
\begin{eqnarray*}
\left|J_{1,4}^n(\xi_{u,t})\right|&=&\left|\sum_{j=1}^{n-1}\int_{t_{j-1}}^{t_j}(\bar{f}^j-f,\xi_{u,t})\,ds
+\int_{t_{n-1}}^{t}(\bar{f}^n-f,\xi_{u,t})\,ds\right|\\
&\leq&\left(\sum_{j=1}^{n-1}\int_{t_{j-1}}^{t_j}\|\bar{f}^j-f\|\,ds
+\int_{t_{n-1}}^{t}\|\bar{f}^n-f\|\,ds\right) E_1(t^*).
\end{eqnarray*}
On substituting the estimates of $J_{1,j}^n(\xi_{u,t})$, $j=1,\cdots,4$, in (\ref{3.26}), we arrive at
$$
 E_1(t)\leq E_1(t^*)
 \leq E_1(0)+2\sum_{j=1}^{4}{\mathcal  E}_{1,j}(t)+2\int_0^t\|\eta_{u,tt}(s)\|\,ds.
$$
This completes the rest of the proof.$\hfill \Box$

\begin{remark} The term $(\br_2^j-\nabla\cdot\bS^j)$ can be replaced by 
$(\partial_t^2U^j-\bar{f}^j)$.
\end{remark}

%---------------------------------
%---------------------------------     L^\infty(L^2)  estimate
%--------------------------------- 

%\newpage

For obtaining $L^\infty(L^2)$ estimate for $e_u$, we now integrate
(\ref{error-3}) with respect to time from $0$ to $t$ with $t\in(t_{n-1},t_n]$, 
to arrive at
\begin{eqnarray}\label{3.31}
\left(\xi_{u,t},w\right)+\left(\nabla\cdot\hat{\xi}_{\bs},w\right)
&=& \left(e_{u,t}(0),w\right)+\left(\eta_{u,t},w\right)\\
&&+J_{2,1}^n(w)+J_{2,2}^n(w)+J_{2,3}^n(w)+J_{2,4}^n(w)\nonumber,
\end{eqnarray}
where
\begin{eqnarray*}
J_{2,1}^n(w)&:=&\sum_{j=1}^{n-1}\int_{t_{j-1}}^{t_j}\left((I-P_h^j)U_{tt},w\right)ds
+\int_{t_{n-1}}^{t}\left((I-P_h^n)U_{tt},w\right)ds,\\
J_{2,2}^n(w)&:=&\sum_{j=1}^{n-1}\int_{t_{j-1}}^{t_j}\mu^j(\partial_t^2U^j,w)\,ds
+\int_{t_{n-1}}^{t}\mu^n(\partial_t^2U^n,w)\,ds,\\
J_{2,3}^n(w)&:=&\sum_{j=1}^{n-1}\int_{t_{j-1}}^{t_j}(\nabla\cdot(\tilde{\bs}-\tilde{\bs}^j),w)\,ds
+\int_{t_{n-1}}^{t}(\nabla\cdot(\tilde{\bs}-\tilde{\bs}^n),w)\,ds,\\
J_{2,4}^n(w)&:=&\sum_{j=1}^{n-1}\int_{t_{j-1}}^{t_j}(\bar{f}^j-f,w)\,ds
+\int_{t_{n-1}}^{t}(\bar{f}^n-f,w)\,ds.
\end{eqnarray*}
Note that $$J_{2,4}= \int_{t_{n-1}}^{t}(\bar{f}^n-f,w)\,ds \quad \text{ and }\quad J_{2,2}^n(w)= \int_{t_{n-1}}^{t}\mu^n(\partial_t^2U^n,w)\,ds$$
as $\int_{t_{j-1}}^{t_j}\mu^j=0$. Further, since $P_h^j$ commutes with time derivative, 
$J_{2,1}^n(w)$ can be written as
\begin{eqnarray}\label{3.32}
J_{2,1}^n(w)&=&\sum_{j=1}^{n-1}\left( (I-P_h^j)U_t(t_j)-
(I-P_h^j)U_t(t_{j-1}),w\right)\nonumber\\
&&+\left((I-P_h^n)U_t(t),w\right)-\left((I-P_h^n)U_t(t_{n-1}),w\right)\nonumber\\
&=&\sum_{j=1}^{n-1}\left( (I-P_h^{j})U_t(t_j)-
(I-P_h^{j-1})U_t(t_{j-1}),w\right)+\sum_{j=1}^{n-1}(P_h^{j}-P_h^{j-1})U_t(t_{j-1})\nonumber\\
&&+\left((I-P_h^n)U_t(t),w\right)-\left((I-P_h^n)U_t(t_{n-1}),w\right)\nonumber\\
&=&\sum_{j=0}^{n-1}\left((P_h^{j+1}-P_h^j)U_t(t_j),w\right)-
\left((I-P_h^0)U_t(0),w\right) \\
&&+\left((I-P_h^n)U_t(t),w\right).\nonumber
\end{eqnarray}
Below, we prove one of the main results of this section.
\begin{theorem}\label{thm-L2}
Let $(u,\bs)$ and $(U,\bS)$ be the solution of $(\ref{bsmf})$-$(\ref{umf})$  
and $(\ref{bS-U0})$-$(\ref{bS-U1}),$ respectively. Then, for $t\in (t_{n-1},t_n],$ 
the following estimate holds
$$
\|\xi_{u}(t)\|\leq\|\xi_{u}(0)\|
+2\sum_{j=1}^4{\mathcal  E}_{2,j}(t)+2\int_0^t\|\eta_{u,t}(s)\|\,ds,
$$
where ${\mathcal  E}_{2,1},\cdots, {\mathcal  E}_{2,4}$ will be given in the following proof. 
%\begin{eqnarray}
%\eta_{2,1}(t)&=& \mbox{see below} \nonumber\\
%\eta_{2,2}(t)&=& \nonumber\\
%\eta_{2,3}(t)&=& \nonumber\\
%\eta_{2,4}(t)&=& \nonumber
%\end{eqnarray}
\end{theorem}
\noindent
{\it Proof }. Choose $w=\xi_u$ in (\ref{3.31}) and $v=\hat{\xi}_{\bs}$ in
(\ref{error-4i}).
Then, add the resulting equations to obtain
$$
\frac{1}{2}\frac{d}{dt}\left(\|\xi_{u}\|^2+\|\alpha^{1/2}\hat\xi_{\bs}\|^2\right)
=
(e_{u,t}(0),\xi_{u})+(\eta_{u,t},\xi_{u})+\sum_{j=1}^4J_{2,j}(\xi_{u}).
$$
On integrating from $0$ to $t$ with $t\in(t_{n-1},t_n]$, it follows that
\begin{equation}\label{aa3}
\|\xi_{u}\|^2+\|\alpha^{1/2}\hat\xi_{\bs}\|^2
=\|\xi_{u}(0)\|^2+2\left(e_{u,t}(0),\xi_u\right)+2\int_0^t\left(\eta_{u,t},\xi_u\right)ds
+2\sum_{j=1}^4K_j(t),
\end{equation}
where $K_j(t) = \displaystyle \int_0^tJ_{2,j}(\xi_{u})\,ds$. 
Set
$$
E_2^2(t): =  \|\xi_u(t)\|^2+\|\alpha^{1/2}\hat\xi_{\bs}(t)\|^2,
$$
and let $t=t^{**} \in (0,t]$ be such that
$$
E_2(t^{**}) =\displaystyle \max_{0 \leq s \leq t}  E_2(s).
$$
Then, for $t_{n-1}<t\leq t_n$, we obtain from (\ref{3.32})
\begin{eqnarray*}
K_1(t)&=&
\sum_{j=1}^{n-1}\int_{t_{j-1}}^{t_j}
\sum_{l=0}^{j-1}\left((P_h^{l+1}-P_h^l)U_t(t_l),\xi_u(s)\right)ds\\
&&+\int_{t_{n-1}}^{t}
\sum_{l=0}^{n-1}\left((P_h^{l+1}-P_h^l)U_t(t_l),\xi_u(s)\right)ds
-\int_{0}^{t}\left((I-P_h^0)U_t(0),\xi_u(s)\right)ds\\
&&+\sum_{j=1}^{n-1}\int_{t_{j-1}}^{t_j}\left((I-P_h^j)U_t(s),\xi_u(s)\right)ds
+\int_{t_{n-1}}^{t}\left((I-P_h^n)U_t(s),\xi_u(s)\right)ds.
\end{eqnarray*}
Hence, 
\begin{eqnarray*}
\left|K_1(t)\right|&\leq&\left[
\sum_{j=1}^{n-1}\left(k_j
\sum_{l=0}^{j-1}\|(P_h^{l+1}-P_h^l)\partial_tU^l\|+
(t-t_{n-1})\|(P_h^{j+1}-P_h^j)\partial_tU^j\| \right.\right. \\
&&+\left.\int_{t_{j-1}}^{t_j}\|(I-P_h^j)U_t(s)\|ds\right)+(t-t_{n-1})\|(P_h^1-P_h^0)\partial_tU^0\|\\
&&\left. +t
\|(I-P_h^0)\partial_tU^0\|+\int_{t_{n-1}}^{t}\|(I-P_h^n)U_t(s)\|ds\right]E_2(t^{**})\\
&=:&{\mathcal  E}_{2,1}(t)E_2(t^{**}).
\end{eqnarray*}
The second term can be written as
\begin{eqnarray*}
K_2(t)&=&
\sum_{j=1}^{n-1}\int_{t_{j-1}}^{t_j}
\left(\int_{t_{j-1}}^s\mu^j(\tau)(\partial_t^2U^j,\xi_u(s))d\tau\right)ds\\
&&+\int_{t_{n-1}}^{t}
\left(\int_{t_{n-1}}^s\mu^n(\tau)(\partial_t^2U^n,\xi_u(s))d\tau\right)ds
\end{eqnarray*}
for $t_{n-1}<t\leq t_n$. Since 
$$\displaystyle \int_{t_{j-1}}^s\mu^j(\tau)d\tau=-3k_j^{-1}
\left[(s-t_{j-1/2})^2-\frac{k_j^2}{4}\right],$$ 
we obtain
\begin{eqnarray*}
\left|K_2(t)\right|&=&
\left|-3\sum_{j=1}^{n-1}k_j^{-1}\int_{t_{j-1}}^{t_j}(s-t_{j-1/2})^2(\partial_t^2U^j,\xi_u(s))ds
+\frac{3}{4}\sum_{j=1}^{n-1}\int_{t_{j-1}}^{t_j}k_j(\partial_t^2U^j,\xi_u(s))ds\right.\\
&&\left. -3k_n^{-1}\int_{t_{n-1}}^{t}(s-t_{n-1/2})^2(\partial_t^2U^n,\xi_u(s))ds
+\frac{3}{4}k_n\int_{t_{n-1}}^{t}(\partial_t^2U^n,\xi_u(s))ds\right|\\
&\leq&\left[\sum_{j=1}^{n-1}\|k_j^2\partial_t^2U^j\|
+\left\|k_n^{-1}\left((t-t_{n-1/2})^3+\frac{3k_n^2}{4}(t-t_{n-1})+\frac{k_n^3}{8}
\right)\partial_t^2U^n\right\|\right]E_2(t^{**})\\
&=:&{\mathcal  E}_{2,2}(t)E_2(t^{**}).
\end{eqnarray*}
%
%&\leq&\left[\frac{1}{8}\sum_{j=1}^{n-1}\|k_j^2\partial_t^2U^j\|
%+\frac{3}{4}\sum_{j=1}^{n-1}\|k_j^2\partial_t^2U^j\|
%+\left\|k_n^{-1}\left((t-t_{n-1/2})^3+\frac{k_n^3}{8}\right)\partial_t^2U^n\right\|\right.\\
%&&\;\;\left.+\frac{3}{4}\|k_n(t-t_{n-1})\partial_t^2U^n\|\right]{\mathscr{E}}_2(t^{**}).
%
%
For the third term $K_3(t)$, we obtain for $t_{n-1}<t\leq t_n$ and $t_{j-1}<s\leq t_j$,
\begin{eqnarray*}
K_3(t)&=&
\sum_{j=1}^{n-1}\int_{t_{j-1}}^{t_j}\sum_{l=1}^{j-1}\int_{t_{l-1}}^{t_l}
\left(\nabla\cdot(\tilde{\bs}(\tau)-\tilde{\bs}^l),\xi_u(s)\right)d\tau ds\\
&&+\int_{t_{n-1}}^{t}\sum_{l=1}^{n-1}\int_{t_{l-1}}^{t_l}
\left(\nabla\cdot(\tilde{\bs}(\tau)-\tilde{\bs}^l),\xi_u(s)\right)d\tau ds\\
&&+\sum_{j=1}^{n-1}\int_{t_{j-1}}^{t_j}\int_{t_{j-1}}^{s}
\left(\nabla\cdot(\tilde{\bs}(\tau)-\tilde{\bs}^j),\xi_u(s)\right)d\tau ds\\
&&+\int_{t_{n-1}}^{t}\int_{t_{n-1}}^{s}
\left(\nabla\cdot(\tilde{\bs}(\tau)-\tilde{\bs}^n),\xi_u(s)\right)d\tau ds\\
&=&I_1+I_2+I_3+I_4.
\end{eqnarray*}
Using (\ref{3.28}) and the fact that 
$\displaystyle \int_{t_{l-1}}^{t_l}\left(k_l^{-1}(t_l-\tau)^3-(t_l-\tau)^2\right)
d\tau=-\frac{k_l^3}{12},$  we find that
\begin{eqnarray*}
I_1&=&
\sum_{j=1}^{n-1}\int_{t_{j-1}}^{t_j}\sum_{l=1}^{j-1}\frac{k_l^2}{2}
\left(\nabla\cdot\partial_t\tilde{\bs}^l,\xi_u(s)\right)ds
+\sum_{j=1}^{n-1}\int_{t_{j-1}}^{t_j}\sum_{l=1}^{j-1}\frac{k_l^3}{12}
\left(\nabla\cdot\partial_t^2\tilde{\bs}^l,\xi_u(s)\right)ds,\\
%&\leq& \left[\frac{1}{2} \sum_{j=1}^{n-1}k_j\sum_{l=1}^{j-1}k_l^2
%\|\nabla\cdot\partial_t\tilde{\bs}^l\|
%+\sum_{j=1}^{n-1}k_j\sum_{l=1}^{j-1}\frac{k_l^3}{12}
%\|\nabla\cdot\partial_t^2\tilde{\bs}^l\| \right]{\mathscr{E}}_2(t^{**}).\\
|I_1|&\leq& \left[\sum_{j=1}^{n-1}k_j\sum_{l=1}^{j-1}\left(\frac{k_l^2}{2}
\|\nabla\cdot\partial_t\tilde{\bs}^l\|
+\frac{k_l^3}{12}
\|\nabla\cdot\partial_t^2\tilde{\bs}^l\|\right) \right] E_2(t^{**}).\\
\end{eqnarray*}
Similarly,
\begin{eqnarray*}
|I_2|
&\leq&\left[ (t-t_{n-1}) \sum_{j=1}^{n-1}\left(\frac{k_j^2}{2}
\|\nabla\cdot\partial_t\tilde{\bs}^j\|
+\frac{k_j^3}{12}
\|\nabla\cdot\partial_t^2\tilde{\bs}^j\|\right)\right] E_2(t^{**}).
\end{eqnarray*}
For the $I_3$ and $I_4$ terms, we easily obtain 
\begin{eqnarray*}
|I_3|&\leq&  \left[
\sum_{j=1}^{n-1} \int_{t_{j-1}}^{t_j}\left(\frac{(t_j-s)^2}{2}-\frac{k_j^2}{2}\right)
\|\nabla\cdot\partial_t\tilde{\bs}^j\|\,ds\right. \\
&& +\left.
\sum_{j=1}^{n-1} \int_{t_{j-1}}^{t_j}\left(\frac{(t_j-s)^3}{3}-k_j^{-1}\frac{(t_j-s)^4}{4}\right)
\|\nabla\cdot\partial_t^2\tilde{\bs}^j\|\,ds\right] E_2(t^{**})\\
&\leq&  \left[\sum_{j=1}^{n-1}\left(\frac{k_j^3}{3}
\|\nabla\cdot\partial_t\tilde{\bs}^j\|+
\frac{k_j^4}{20}\|\nabla\cdot\partial_t^2\tilde{\bs}^j\|\right)\right] E_2(t^{**}),
\end{eqnarray*}
and
\begin{eqnarray*}
|I_4|&\leq&  (t-t_{n-1})\left[
\frac{k_n^3}{3}
\|\nabla\cdot\partial_t\tilde{\bs}^n\|
+\frac{k_n^4}{20}
\|\nabla\cdot\partial_t^2\tilde{\bs}^n\|\right] E_2(t^{**}).
\end{eqnarray*}
Now collect terms,  replace $\nabla\cdot\tilde{\bs}^j$ 
by $-\br_2^j+\nabla\cdot\bS^j$ using (\ref{3.29}), and set ${\mathcal  E}_{2,3}:=M_1+M_2,$ where  
\begin{eqnarray*}
M_1&=& \sum_{j=1}^{n-1}\left[k_j\left(\sum_{l=1}^{j-1}\frac{k_l^2}{2}
\|\partial_t(\br_2^l-\nabla\cdot\bS^l)\|\right)+
\left((t-t_{n-1})\frac{k_j^2}{2}+\frac{k_j^3}{3}\right)
\| \partial_t(\br_2^j-\nabla\cdot\bS^j)\|\right]\\
&&+(t-t_{n-1})\frac{k_n^3}{3}\|\partial_t(\br_2^n-\nabla\cdot\bS^n)\|,
\end{eqnarray*}
and
\begin{eqnarray*}
M_2&=& \sum_{j=1}^{n-1}\left[k_j\left(\sum_{l=1}^{j-1}\frac{k_l^3}{12}
\|\partial_t^2(\br_2^l-\nabla\cdot\bS^l)\|\right)+
\left((t-t_{n-1})\frac{k_j^3}{12}+\frac{k_j^4}{20}\right)
\|\partial_t^2(\br_2^j-\nabla\cdot\bS^j)\|\right]\\
&&+\frac{k_n^4}{20}(t-t_{n-1})\|\partial_t^2(\br_2^n-\nabla\cdot\bS^n)\|
\end{eqnarray*}
so that $$\left|K_3(t)\right| \leq {\mathcal  E}_{2,3}(t) E_2(t^{**}).$$
For the last term $K_4(t)$, one can repeat previous arguments to arrive at
\begin{eqnarray*}
\left|K_4(t)\right|&\leq&\left|
\sum_{j=1}^{n-1}k_j\int_{t_{j-1}}^{t_j}\|\bar{f}^j-f(\tau)\|\,d\tau 
+(t-t_{n-1})\int_{t_{n-1}}^{t}\|\bar{f}^n-f(\tau)\|\,d\tau \right| E_2(t^{**})\\
&=:&{\mathcal  E}_{2,4}(t) E_2(t^{**}).
\end{eqnarray*}
% \begin{eqnarray*}
%\left|\int_0^tJ_{2,4}(\xi_u(s))\,ds\right|&=&\left|
%\sum_{j=1}^{n-1}\left(k_j\sum_{l=1}^{j-1}\int_{t_{l-1}}^{t_l}\|f^l-f(\tau)\|\,d\tau+
%(t-t_{n-1})\int_{t_{j-1}}^{t_j}\|\bar{f}^j-f(\tau)\|\,d\tau\right.\right.\\
%&&\left.\left.+k_j\int_{t_{j-1}}^{t_j}\|\bar{f}^j-f(\tau)\|\,d\tau \right)
%+(t-t_{n-1})\int_{t_{n-1}}^{t}\|\bar{f}^n-f(\tau)\|\,d\tau \right| E_2(t^{**})\\
%&=:&{\mathcal  E}_{2,4}(t) E_2(t^{**}).
%\end{eqnarray*}
On substituting in (\ref{aa3}), it follows that
$$
\|\xi_{u}(t)\|\leq\|\xi_{u}(0)\|
+2\sum_{j=1}^4{\mathcal  E}_{2,j}(t)+2\int_0^t\|\eta_{u,t}(s)\|\,ds,
$$
which completes the rest of the proof.  $\quad\hfill \Box$
%\newpage 

In order to present  the  final theorem in this paper, 
we introduce some notations: For $ D := \Omega ~\mbox{or}~K $, let
\begin{eqnarray*}
{\mathcal E}_1^0(D) & = &  
\|h_0 (\alpha \bS^0+ \nabla_h U^{0})\|_{L^2(D)}, \; \nonumber \\
{\mathcal E}_2^n(D) & = &   \Big(\|h_n^{\ell+1} \br_2^n\|_{L^2(D)}  + 
\|h_n(\alpha \bS^n + \nabla_h U^n)\|_{L^2(D)}
\Big), \nonumber \\
{\mathcal E}_3^n(D) & = &  \Big(\|h_n^{\ell+1} \partial_t\br_2^n\|_{L^2(D)}
+  \|h_n \partial_t (\alpha \bS^n + \nabla_h U^n)\|_{L^2(D)}  \Big),\nonumber \\
{\mathcal E}_4^0 (D)& = & \|h_0 (\alpha \partial_t\bS^0 + \nabla_h \partial_tU^0)\|_{L^2(D)},
\nonumber \\
{\mathcal E}_5^0(D) & = & \Big( 
 \| h^{1/2}_0 J(\alpha \bS^0\cdot {\bf t})\|_{0,\Gamma_h,D}
+ \|h_0 \,{\mbox {curl }}_h\,(\alpha \bS^0)\|_{L^2(D)}\Big), \nonumber  \\
{\mathcal E}_6^n(D) & = & \Big( \|h_n \br^n_2 
\|_{L^2(D)}  + \| h_n^{1/2} J(\alpha \bS^n \cdot
{\bf t})\|_{0,\Gamma_h,D}
+ \|h_n \,{\mbox {curl}}_h\,(\alpha \bS^n)\|_{L^2(D)}  \Big), \nonumber\\
{\mathcal E}_7^n(D) & = &  \Big(\|h_n^{\ell+1} \partial_t\br_2^n\|_{L^2(D)}
+  \|h_n \partial_t (\alpha \bS^n + \nabla_h U^n)\|_{L^2(D)}  \Big),\nonumber \\
{\mathcal E}_8^n(D) & = & \Big(\|h_n^{\ell+1} \partial_t^2\br_2^n\|_{L^2(D)}
+  \|h_n \partial_t^2 (\alpha \bS^n + \nabla_h U^n)\|_{L^2(D)}  \Big),
\end{eqnarray*}
%
%and
%
\begin{eqnarray*}
{\mathcal  E}^n_{1,1}&=&\sum_{j=1}^{n}\int_{t_{j-1}}^{t_j}||(1+\mu^j)(I-P_h^j)\partial_t^2 U^j||\,ds,\nonumber\\
{\mathcal  E}^n_{1,2}&=& \sum_{j=1}^{n}\int_{t_{j-1}}^{t_j}||\mu^j\partial_t^2 U^j||\,ds,\nonumber\\
{\mathcal  E}^n_{1,3}&=&\sum_{j=1}^{n}
\left(\frac{k_j^2}{2}||\partial_t(\br_2^j-\nabla\cdot\bS^j)||+
\frac{k_j^3}{12}||\partial_t^2(\br_2^j-\nabla\cdot\bS^j)||\right),
\nonumber\\
{\mathcal  E}^n_{1,4}&=&\sum_{j=1}^{n}\int_{t_{j-1}}^{t_j}||\bar{f}^j-f(s)||\,ds,
\end{eqnarray*}
and
\begin{eqnarray*}
{\mathcal  E}^n_{2,1}&=&\sum_{j=1}^{n}\left(k_j
\sum_{l=0}^{j-1}\|(P_h^{l+1}-P_h^l)\partial_tU^l\|+\int_{t_{j-1}}^{t_j}\|(I-P_h^j)
U_t(s)\|\,ds\right)\\
&&+k_n\|(P_h^1-P_h^0)\partial_tU^0\|+t^n\|(I-P_h^0)\partial_tU^0\|,\\
{\mathcal  E}^n_{2,2}&=&\sum_{j=1}^{n}k_j^2||\partial_t^2 U^j||,\\
{\mathcal  E}^n_{2,3}&=&
\sum_{j=1}^{n}k_j\sum_{l=1}^{j-1}\left(\frac{k_l^2}{2}
\|\partial_t(\br_2^l-\nabla\cdot\bS^l)+\frac{k_l^3}{12}
\|\partial_t^2(\br_2^l-\nabla\cdot\bS^l)\|\right),\\
{\mathcal  E}^n_{2,4}&=&\sum_{j=1}^{n}k_j\int_{t_{j-1}}^{t_j}\|\bar{f}^j-f(s)\|\,ds.
%{\mathcal  E}^n_{2,4}&=&\sum_{j=1}^{n}\left(k_j\sum_{l=1}^{j-1}\int_{t_{l-1}}^{t_l}\|f^l-f(\tau)\|d\tau+
%+k_j\int_{t_{j-1}}^{t_j}\|\bar{f}^j-f(\tau)\|d\tau \right).
\end{eqnarray*}
Using estimates of $\eta_{u}$ and $\eta_{\bs}$ in Theorems  
\ref{xius-thm} and \ref{thm-L2},  the final theorem of this section can be written as
\begin{theorem}\label{final_thm}
Let $(u,\bs)$  be the solution of $(\ref{bsmf})$-$(\ref{umf})$ and $(U,\bS)$ be
the solution of $(\ref{bS-U0})$-$(\ref{bS-U1})$. Then for $m \in [1;N],$ the
following estimates hold for the RT index $\ell=0,1$:
\begin{equation}\label{xiu-3}
\|U^m-u(t_m)\| \leq \|e_{u}(0)\| + C_1 {\mathcal E}_1^0(\Omega) + 
C_2 {\mathcal E}_2^m(\Omega) + C_{3}\sum_{n=1}^{m} k_n {\mathcal E}_3^n
+\sum_{i=1}^4 c_i {\mathcal  E}^n_{2,i}(\Omega),
\end{equation}
and
\begin{eqnarray}\label{xibs-3}
\|\bS^m-\bs(t_m))\|_{A^{-1}}
&\leq &\|e_{u,t}(0)\| +\|e_{\bs}(0)\|_{A^{-1}}+ C_4 {\mathcal E}_4^0(\Omega) 
+C_5 {\mathcal E}_5^0(\Omega) \nonumber\\
&& +C_6 {\mathcal E}_6^m(\Omega)+ C_7 {\mathcal E}_7^m(\Omega)  +C_{8}\sum_{n=1}^{m} k_n {\mathcal E}_8^n(\Omega)
+\sum_{i=1}^4 c_i {\mathcal  E}^n_{1,i}(\Omega),
\end{eqnarray}
where $C_i$'s and $c_i$'s are constants which depend only on the coefficient matrix $A$, the domain $\Omega$, the shape regularity of the elements, polynomial degree $\ell$ and interpolation constants.
\end{theorem}

\begin{remark} \label{remark4.2} The last term in $(\ref{xibs-3})$, that is, ${\mathcal  E}^n_{1,4}(\Omega)$ (also in $(\ref{xiu-3})$, that is, ${\mathcal  E}^n_{2,4}(\Omega)$) measures the effect of  approximating the forcing function $f$ at discrete points in time. This bound is of similar form as the bound in \cite{LM-06}-\cite{MNP-2012}, where the same discretization has been used in the context of parabolic problems. However, a modification of $\bar{f}^n$ as 
$$
\bar{f}^n = \frac{1}{k_n}\int_{k_{n-1}}^{t_n}f(s)\,ds
$$
will specially improve the estimate ${\mathcal  E}^n_{1,4}(\Omega)$ in $(\ref{xibs-3})$.
\end{remark}

\begin{remark}
The numerical implementation of the proposed  {\it a posteriori} estimators in 
the adaptive algorithm deserves special attention and will be considered elsewhere.
\end{remark}

\section{Conclusion} \label{sec:5}
 The current work presents a first step towards  true {\it a posteriori} estimate in the $L^\infty(L^2)$-norm 
for  mixed finite element approximations of second order wave equations. %The derived bounds appear to be
%of optimal order. 
While Baker's technique is usually used to 
derive $L^\infty(L^2)$ estimates for the displacement $u$, in this paper, we resort to an 
application of integration for deriving these estimates. For adaptive algorithm, we need efficiency bounds, which
%Although no efficiency bounds are discussed in this article, this 
would be an interesting direction for
further research. Moreover, the numerical implementation of the adaptive algorithm
based on the proposed  estimators will be a part of our future work.

% put your thanks here
\section*{Acknowledgments}
The first author
acknowledges the research support by Sultan Qaboos University under Grant IG/SCI/DOMS/13/02.
The second author acknowledges the research
support of the Department of Science and Technology, Government of India
through the National Programme on Differential Equations: Theory, Computation 
and Applications vide DST Project No.SERB/F/1279/2011-2012.

\end{document}